\documentclass[12pt]{amsproc}
  \usepackage{latexsym} 
  \usepackage[all]{xy}
  \usepackage{amsfonts} 
  \usepackage{amsthm} 
  \usepackage{amsmath} 
  \usepackage{amssymb}
  \usepackage{pifont}  
  \usepackage{enumerate}
  \xyoption{2cell}

\newcommand{\cC}{{C}}

  \def\sw#1{{\sb{(#1)}}} 
   
  \def\su#1{{\sp{[#1]}}}  
  \def\suc#1{{\sp{(#1)}}}

  \def\<{{\langle}} 
  \def\>{{\rangle}}

  \def\eps{\varepsilon}

  \def\note#1{{}}

  \def\note#1{} 

  \def\cM{{\mathcal M}} 
   
  \def\cN{{\mathcal N}} 
    
  \def\cC{{\mathcal C}} 
  \def\cA{{\mathcal A}} 
  \def\cB{{\mathcal B}} 
  \def\cD{{\mathcal D}} 
  \def\cE{{\mathcal E}}

  \def\rhom#1#2#3{{{\rm Hom}\sb{#1}\left(#2,#3\right)}}

  \def\rend#1#2{{{\rm End}\sb{#1}(#2)}}

  \def\Rhom#1#2#3{{{\rm Hom}\sp{#1}(#2,#3)}}

  \def\beq{\begin{equation}} 
  \def\eeq{\end{equation}}

  \def\id{\mathrm{id}}

  \def\ot{{\otimes}}

  \def\roM{\varrho^{M}} 
   \def\roN{\varrho^{N}} 
    \def\roV{\varrho^{V}}

  \newcounter{zlist} 
  \newenvironment{zlist}{\begin{list}{(\arabic{zlist})}{ 
  \usecounter{zlist}\leftmargin2.5em\labelwidth2em\labelsep0.5em 
  \topsep0.6ex
  \parsep0.3ex plus0.2ex minus0.1ex}}{\end{list}}

  \newcounter{blist} 
  \newenvironment{blist}{\begin{list}{(\alph{blist})}{ 
  \usecounter{blist}\leftmargin2.5em\labelwidth2em\labelsep0.5em 
  \topsep0.6ex 
  \parsep0.3ex plus0.2ex minus0.1ex}}{\end{list}} 

  \newcounter{rlist}

\def\stac#1{\raise-.2cm\hbox{$\stackrel{\displaystyle\otimes}{\scriptscriptstyle{#1}}$}}

\def\cten#1{\raise-.2cm\hbox{$\stackrel{\displaystyle\widehat{\otimes}}
{\scriptscriptstyle{#1}}$}}

  \def\Label#1{\label{#1}\ifmmode\llap{[#1] }\else 
  \marginpar{\smash{\hbox{\tiny [#1]}}}\fi} 
  \def\Label{\label}

  \newtheorem{proposition}{Proposition}[section]
  \newtheorem{lemma}[proposition]{Lemma} 
  \newtheorem{corollary}[proposition]{Corollary} 
  \newtheorem{theorem}[proposition]{Theorem} 

  \theoremstyle{definition} 
  \newtheorem{definition}[proposition]{Definition}
    
  \newtheorem{example}[proposition]{Example}

  \theoremstyle{remark} 
  \newtheorem{remark}[proposition]{Remark}

  \newcounter{c} 
   
  \newcommand{\etyk}[1]{\vspace{-7.4mm}$$\begin{equation}\Label{#1} 
  \addtocounter{c}{1}} 
  \renewcommand{\]}{\ifnum \value{c}=1 $$\else \end{equation}\fi} 
\newcommand{\cov}{{\rm cov}}

\def\ot{\otimes}

\def\KK{{\mathbb K}}

\def\NN{{\mathbb N}}

\def\ZZ{{\mathbb Z}}

\newcommand{\Cc}{\mathcal{C}}

\newcommand{\Pp}{\mathcal{P}}

\newcommand{\Ss}{\mathcal{S}}
\newcommand{\Tt}{\mathcal{T}}

\def\cov{\stackrel{\rightarrow}{\nabla}}
\def\voc{\stackrel{\leftarrow}{\nabla}}

\def\*C{{}^*\hspace*{-1pt}{\Cc}}

\def\text#1{{\rm {\rm #1}}}

\def\mod{\mathbf{Mod}}

 \def\1{\mathbf{1}}

 \def\cdga{\mathbf{CDGA}}
  \def\sfcdga{\mathbf{sf\,CDGA}}
    \def\crg{\mathbf{Crg}}
        \def\bcrg{\mathbf{bCrg}}
      \def\scrg{\mathbf{sCrg}}
\def\HoA{$\mathbf{Ho}(\mod$-$\cA)$}
\def\com#1{$\mathbf{com}$-$#1$}

\begin{document} 

\title[Curved differential graded algebras and corings]{Curved differential graded algebras and corings} 
 \author{Tomasz Brzezi\'nski}
 \address{ Department of Mathematics, Swansea University, 
  Singleton Park, \newline\indent  Swansea SA2 8PP, U.K.} 
  \email{T.Brzezinski@swansea.ac.uk}   

    \date{January 2013} 
  \subjclass[2010]{16T15; 16E45; 18E30} 
  \begin{abstract} 
A relationship between curved differential algebras and corings is established and explored. In particular it is shown that the category of semi-free curved differential graded algebras is equivalent to the category of corings with surjective counits. Under this equivalence, comodules over a coring correspond to integrable connections or quasi-cohesive curved modules, while contramodules over a coring correspond to   a specific class of curved modules introduced  and termed $\ZZ$-divergences in here. 
\end{abstract} 
  \maketitle

\section{Introduction}
It is well known that there is a one-to-one correspondence between corings with a grouplike element and semi-free differential graded algebras; see \cite{Roi:mat}. Under this correspondence comodules of a coring coincide with modules with flat connections over the induced differential graded algebra; see \cite{Brz:gro}. The first aim of this note is to extend the above relationship to any corings with a distinguished element (not necessarily a grouplike element) and thus interpret them as {\em curved} differential graded algebras of a particular kind. The second (less-direct) aim is to explore potential for interaction between these two points of view on equivalent mathematical objects. We make the first few steps in direction of this objective by associating curved differential graded algebras to progenerators, partial orders and matrices, and introducing a new class of curved differential graded modules.

The history of corings goes back to at least Sweedler's paper \cite{Swe:pre}, while curved differential graded algebras are a much more recent invention; see \cite{GetJon:alg} and \cite{Pos:non}. The former attracted some attention in algebra and ring theory (see \cite{Brz:com} for a review), the latter start playing an increasingly significant role in non-commutative geometry \cite{Sch:sup}, \cite{Blo:dua},  \cite{Blo:dua2} and mathematical physics, in particular in analysis of Landau-Ginzburg models; see, for example \cite{Laz:gra}, \cite{CalTu:cur}, \cite{Seg:clo}. Both are arrived at by relaxing requirements that previously have been considered to be natural. In the case of corings one studies comultiplication on bimodules rather than vector spaces or modules over a commutative ring. In the case of curved differential graded algebras one relaxes the nilpotency of the derivation. In Section~2 we collect basic information about curved differential graded algebras and their modules and about corings and their comodules and contramodules. 

The main results of the paper are contained in Sections 3 and 4 whose contents we summarise presently. The aim of Section 3 is to establish an equivalence of categories of semi-free curved differential graded algebras to the category of corings with surjective counits. This is achieved in two steps. First, we describe a construction which to every coring with a chosen element associates a semi-free curved differential algebra. This construction is functorial and it yields an object in the category of {\em curved} differential graded algebras isomorphic to a genuine (i.e.\ not curved) differential graded algebra.  On the other hand, if the value of the counit on the chosen element is one (such an element is called a {\em base point} and the pair: coring and element is termed a {\em based coring}), then another semi-free curved differential graded algebra can be constructed. The latter construction admits a reciprocal, which is described in the second part of Section~3, and yields the aforementioned equivalence of categories. An interesting observation here is that the constructed curved differential graded algebra is isomorphic to a differential graded algebra if, and only if, the coring has a group-like element. We show also that, under this correspondence, complexes of comodules of a coring yield integrable connections or {\em quasi-cohesive modules} over the associated curved differential graded algebra and vice versa.  Exploring this correspondence we introduce the differential graded category of {\em comodule complexes} and reveal the triangulated structure of its homotopy category.

The relationship between comodules and integrable connections raises a natural question: 
what specific curved modules correspond to contramodules on the differential algebra side? We answer this question in Section~4 by introducing the notion of an {\em integrable $\ZZ$-divergence} which is a curved differential graded algebra version of the notion of a hom-connection from \cite{Brz:con}, which, in turn, is a noncommutative counterpart of that of a {\em right connection} introduced by Manin in context of supermanifolds \cite[Chapter~4\S 5]{Man:gau}. We relate contramodules with integrable $\ZZ$-divergences and show that $\ZZ$-divergences supplement connections in the following way: An integrable connection concentrated in one degree yields a non-negatively graded curved module while an integrable $\ZZ$-divergence concentrated in one degree yields a non-positively graded curved module.

All algebras in this paper are associative with identity and over a base commutative ring $\KK$. The degree of a homogeneous element $m$ of a graded $\KK$-module $M^\bullet=\bigoplus_{n\in \ZZ} M^n$ is denoted by $|m|$. We write $A$ for the zero-degree component of a graded algebra $A^\bullet = \bigoplus_{n\in \NN} A^n$, and $[a,b]$ denotes the graded commutator, i.e.\ $[a,b] = ab -(-1)^{|a||b|}ba$. We often write $V$ for the identity morphism of an object $V$.

\section{Preliminaries on curved differential graded algebras and corings}\label{sec.cdga} 
\setcounter{equation}{0}
The aim of this section is to recall definitions and describe basic examples of curved differential graded algebras and their modules, and corings and their co- and contra-modules.
\subsection{Curved differential graded algebras and their modules}
\begin{definition}[\cite{GetJon:alg}, \cite{Pos:non}] A {\em curved differential graded algebra} is a triple $\cA = (A^\bullet, d,\gamma)$ consisting of an $\NN$-graded algebra $A^\bullet = \oplus_{n\in \NN} A^n$, a degree-one $\KK$-linear map $d: A^\bullet \to A^{\bullet +1}$ and an element $\gamma\in A^2$ that satisfy the following conditions:
\begin{blist}
\item $d$ is a graded derivation, i.e.\ it satisfies the graded Leibniz rule;
\item for all $a\in A^\bullet$, 
$$
d\circ d(a) = [\gamma, a];
$$
\item $d \gamma = 0$.
\end{blist}
The element $\gamma$ is referred to as the {\em curvature} of $\cA$ and condition (c) is known as the {\em Bianchi identity}.
\end{definition}
\begin{definition}[\cite{Pos:non}]\label{def.cat}
A morphism of curved differential algebras from $\cA = (A^\bullet, d_A,\gamma_A)$ to $\cB = (B^\bullet, d_B,\gamma_B)$ is a pair $(f, \omega)$ consisting of a morphism $f: A^\bullet \to B^\bullet$ of graded algebras and an element $\omega \in B^1$ such that:
\begin{blist}
\item for all $a\in A$,
$$f(d_Aa) = d_B f(a) +  [\omega, f(a)];$$
\item $f( \gamma_A) = \gamma_B +d_B\omega + \omega^2$.
\end{blist}
The composition of morphisms $(f,\omega):\cA \to \cB$, $(f',\omega'):\cB \to \cC$ is defined as $(f'\circ f, f'(\omega)+\omega')$, and the category of curved differential graded algebras  is denoted by $\cdga$.
\end{definition}

\begin{example}\label{ex.semi-free}
By a {\em semi-free curved differential graded algebra}  we understand a curved differential algebra $\cA$ with $A^\bullet = T_A(V)$, the tensor algebra of an $A$-bimodule $V$. That is
$$
A^0 = A, \qquad A^1 = V, \qquad A^n = V^{\ot _A n}.
$$
A morphism from a semi-free differential graded algebra $\cA$ to any curved differential graded algebra $\cB$ is fully determined by $\omega \in B^1$ and two maps:
\begin{blist}
\item an algebra map $f_0: A^0\to B^0$,
\item an $A$-bimodule map $f_1: A^1 \to B^1$, where the $A$-bimodule structure on $B^1$ is induced through $f_0$.
\end{blist}
We write $T(f_0,f_1)$ for the corresponding morphism of graded algebras, i.e., for all $a\in A$, $v_1,\ldots , v_n\in A^1$,
$$
T(f_0,f_1)(a) = f_0(a), \quad T(f_0,f_1)(v_1\ot_A v_2\ot_A \ldots \ot_A v_n) = f_1(v_1) f_1(v_2) \ldots f_1(v_n).
$$
The full subcategory of $\cdga$ consisting of semi-free curved differential graded algebras is denoted by $\sfcdga$.
\end{example}

\begin{example}\label{ex.sweed}
To any algebra $A$, a semi-free curved differential algebra can be associated by choosing $V= A\ot A$, any $x = \sum_i x^i\ot y^i \in A\ot A$, setting
$$
\gamma = \sum_{i,j} x^i\ot y^ix^j\ot y^j - \sum_i x^i\ot 1\ot y^i \in A\ot A\ot A \cong V\ot _A V,
$$
and defining
\begin{eqnarray*}
d(a_0\ot \ldots \ot a_n) &=& xa_0 \ot \ldots \ot a_n + \sum_{k=1}^n (-1)^k a_0\ot  \ldots  \ot a_{k-1}\ot 1\ot a_k\ot \ldots \ot a_n\\
&& +(-1)^{n+1} a_0\ot \ldots \ot a_nx.
\end{eqnarray*}
\end{example}

\begin{definition}\label{def.mod}
A {\rm right module} over a curved differential graded algebra $\cA = (A^\bullet, d,\gamma)$ or a {\em curved differential graded $\cA$-module} is a pair $\cM = (M^\bullet,d_M)$, consisting of a $\ZZ$-graded module $M^\bullet = \oplus_{n\in \ZZ}M^n$ and a degree-one $\KK$-linear map $d_M: M^\bullet \to M^{\bullet+1}$ such that
\begin{blist}
\item for all $a\in A^\bullet$ and $m\in M^n$,
$$
d_M(ma) = d_M(m) a + (-1)^n md(a);
$$
\item for all $m\in M^\bullet$,
$$
d_M\circ d_M( m) = - m\gamma.
$$
\end{blist}
Morphisms of right $\cA$-modules are simply morphisms of $A^\bullet$-modules. The category $\mod$-$\cA$ of right $\cA$-modules is a differential graded category;  the homomorphisms $\rhom \cA \cM\cN$ are complexes with the coboundary
$$
d_{\cM,\cN} (f)  = d_N\circ f - (-1)^n f\circ d_M,
$$
for all degree-$n$ right $A^\bullet$-linear maps $f: M^\bullet \to N^\bullet$.
\end{definition}

Since $\mod$-$\cA$ is a differential graded category, one can associate to it two other categories both with the same objects as those of $\mod$-$\cA$ but with differently defined morphisms. The morphisms in $Z^0(\mod$-$\cA)$ are the zero-degree closed morphisms of $\mod$-$\cA$. The morphisms in the {\em homotopy category} of $\mod$-$\cA$, from $\cM$ to $\cN$ are elements of the zero cohomology group of $\rhom \cA \cM\cN$. The homotopy category of $\mod$-$\cA$ is denoted by \HoA. We refer the reader to \cite{Kel:dif} for more information on differential graded categories.

\begin{example}\label{ex.connection} An example of a module over $\cA$ is provided by a {\em quasi-cohesive module} or an {\em integrable $\ZZ$-connection} \cite[Definition~3.7]{Blo:dua}. Write $A:= A^0$ and, for any $\ZZ$-graded $A$-module $M^\bullet$ equip $M^\bullet \ot _A A^\bullet$ with the tensor product grading, i.e.\ $|m \ot a| = |m| + |a|$. A pair $\cM = (M^\bullet, \cov)$, consisting of a graded right $A$-module $M^\bullet$ and a degree-one $\KK$ linear endomorphism $\cov$ of  $M^\bullet \ot _A A^\bullet$ is called an {\em integrable $\ZZ$-connection} provided
\begin{blist}
\item for all $a\in A^\bullet$ and homogeneous $x\in M^\bullet \ot _A A^\bullet$,
$$
\cov(xa) = \cov(x) a + (-1)^{|x|} xda;
$$
\item for all $x\in M^\bullet \ot _A A^\bullet$,
$$
\cov\circ \cov (x) = -x\gamma.
$$
\end{blist}
Condition (b) is referred to as {\em integrability} of the $\ZZ$-connection and the map $\cov$ is called a {\em covariant derivative}. An integrable $\ZZ$-connection is an example of a curved differential graded module over $\cA$ with the module part provided by $M^\bullet \ot _A A^\bullet$ and  the differential $\cov$. 

The covariant derivative of an integrable $\ZZ$-connection is fully determined by its value on $M^\bullet = M^\bullet \ot_A A^0$. Writing $\cov^n$ for the component of $\cov$ given as $\cov^n: M^\bullet \mapsto M^{\bullet -n+1} \ot_A A^n$ one easily finds that $(-1)^k\cov^1$ is a (usual) covariant derivative on $M^k$ while $\cov^{n\neq 1}$ are $A$-module maps. Thus to construct a $\ZZ$-connection is the same as to specify suitable maps $\cov^{k,n} : M^k\to M^{k-n+1} \ot_A A^n$, and this is how we will describe $\ZZ$-connections in examples.

A quasi-cohesive module is called a {\em cohesive module} provided $M^\bullet$ is a finitely generated and projective $A$-module and is bounded in both directions as a $\ZZ$-graded module.
\end{example}

\begin{definition}\label{def.bimod}
Let $\cA = (A^\bullet, d_A, \gamma_A)$ and $\cB = (B^\bullet, d_B,\gamma_B)$ be curved differential graded algebras. A curved $(\cA,\cB)$-bimodule is a pair $\cE=(E^\bullet, d_E)$, where $E^\bullet$ is a $\ZZ$-graded $(A^\bullet, B^\bullet)$-bimodule and $d_E: E^\bullet \to E^\bullet$ is a degree one map such that
\begin{blist}
\item for all $a\in A^\bullet$, $b\in B^\bullet$ and $e\in E^\bullet$,
$$
d_E(aeb) = d_A(a)eb + (-1)^{|a|}ad_E(e) b + (-1)^{|a| +|e|} aed_B(b);
$$
\item for all $e\in E^\bullet$,
$$
d_E\circ d_E(e) = \gamma_A e - e\gamma_B.
$$
\end{blist}
\end{definition}

\begin{example}\label{ex.bimod}
$(A^\bullet, d)$ is a curved $(\cA,  \cA)$-bimodule  of any curved differential graded algebra $\cA = (A^\bullet, d, \gamma)$.
\end{example}

\begin{example}\label{ex.bimod2}  
Let $\cM = (M^\bullet, d_M)$ be a curved $\cA$-module. Then the endomorphism ring $S^\bullet = {\rend \cA \cM}^\bullet$ gives rise to the differential graded algebra
$\Ss = (S^\bullet, d_S)$, where $d_S (s) = d_M\circ s - (-1)^{|s|} s\circ d_M$, and $\cM$ is a curved $(\Ss, \cA)$-bimodule with the $\Ss$-action $s m = s(m)$.
\end{example}

Curved bimodules induce curved modules.

\begin{lemma}\label{lem.ind}
Let $\cA = (A^\bullet, d_A, \gamma_A)$ and $\cB = (B^\bullet, d_B,\gamma_B)$ be curved differential graded algebras and let $\cE= (E^\bullet, d_E)$ be a curved $(\cA,\cB)$-bimodule.

\begin{zlist}
\item Given a (right) curved $\cB$-module $\cM=(M^\bullet,d_M)$, denote by $\Xi_B(E,M)^\bullet$ the right $A^\bullet$-module of (internal) graded $B^\bullet$-homomorphisms $E^\bullet \to M^\bullet$. Define,
$$
d_\Xi: \Xi_B(E,M)^\bullet \to \Xi_B(E,M)^\bullet , \qquad \xi \mapsto d_M\circ \xi - (-1)^{|\xi |} \xi \circ d_E.
$$
Then $\Xi_B(\cE,\cM)= (\Xi_B(E,M)^\bullet , d_\Xi)$ is a curved differential graded $\cA$-module.

\item Given a (right) curved $\cA$-module $\cM=(M^\bullet,d_M)$, denote by $T_A(M,E)^\bullet$ the right $B^\bullet$-module given by the tensor product $M^\bullet\ot_A E^\bullet$. Define,
$$
d_{M\ot E}: T_A(M,E)^\bullet \to T_A(M,E)^\bullet , \qquad m\ot_A e \mapsto d_M(m) \ot e + (-1)^{|m |} m \ot  d_E(e).
$$
Then $\Tt_A(\cM,\cE)= (T_A(M,E)^\bullet, d_{M\ot E})$ is a curved differential graded $\cB$-module.
\end{zlist}
\end{lemma}
\begin{proof}
Recall that $\Xi_B(E,M)^i$ consists of all $B^\bullet$-homomorphisms $\xi: E^\bullet \to M^\bullet$ such that, for all $n\in \ZZ$,  $\xi(E^n)\subseteq M^{n+i}$. Even more explicitly, for all $e\in E^n$ and $b\in B^m$,  
$$
\xi(e b) =\xi(e) b \in M^{n+m+i}.
$$
The right $A^\bullet$ multiplication on $\Xi_B(E,M)^\bullet$ is defined by $(\xi a)(e) = \xi(ae)$, and it makes $\Xi_B(E,M)^\bullet$ a graded $A^\bullet$-module, since clearly, for all homogeneous $\xi, a ,e$, $\xi(ae) \in M^{|a| + |e| +|\xi |}$, i.e.\ $|\xi a| = |\xi | + |a|$. 

First we need to check whether $d_\Xi$ is well-defined, i.e.\ whether $d_\Xi(\xi)$ is a graded right $B^\bullet$-module homomorphism. Take any homogeneous $\xi\in \Xi_B(E,M)^\bullet$, $e\in E^\bullet$ and $b\in B^\bullet$ and compute
\begin{eqnarray*}
(d_\Xi\xi)(eb) &=& d_M(\xi(eb)) - (-1)^{|\xi |}\xi( d_E(eb))\\
&=& d_M(\xi(e))b + (-1)^{|\xi(e)|}\xi(e)d_Bb - (-1)^{|\xi |}\xi( d_E(e)b) - (-1)^{|\xi | +|e|}\xi(e)d_Bb\\
&=& (d_\Xi \xi)(e) b,
\end{eqnarray*}
where the Leibniz rule for a curved module, $B^\bullet$-linearity of $\xi$ and the fact that $|\xi(e)| = |\xi | +|e|$ have been used. Since $d_M$ and $d_E$ are degree-one maps, so is $d_\Xi$. Next, to check the Leibniz rule, take homogeneous $\xi\in \Xi_B(E,M)^\bullet$, $e\in E^\bullet$ and $a\in A^\bullet$ and compute
\begin{eqnarray*}
d_\Xi(\xi a)(e) &=& d_M(\xi(ae)) - (-1)^{|\xi | +|a|}\xi( a d_Ee)\\
&=& d_M(\xi(ae)) - (-1)^{|\xi |}\xi( d_E(ae)) + (-1)^{|\xi |}\xi( d_A(a)e)\\
&=& \left(d_\Xi (\xi) a + (-1)^{|\xi |}\xi d_Aa\right)\left(e\right),
\end{eqnarray*}
where the second equality follows by the Leibniz rule for $d_E$. Finally, for the integrability
$$
d_\Xi\circ d_\Xi(\xi)(e) = d_M^2(\xi(e)) - \xi(d_E^2(e)) = -\xi(e)\gamma_B - \xi(\gamma_A e) + \xi(e\gamma_B) = - (\xi \gamma_A)(e),
$$
by the integrability conditions for a curved module, Definition~\ref{def.mod}(b), and for a curved bimodule, Definition~\ref{def.bimod}(b). This proves statement (1). The second statement is proven by similar calculations.
\end{proof}

By Lemma~\ref{lem.ind}(2), since $(A^\bullet, d)$ is a curved $(\cA,\cA)$-bimodule any curved $\cA$-module $(M^\bullet, d_M)$ gives rise to a quasi-cohesive  module $(M^\bullet, d_{M\ot A})$.

\subsection{Based corings,  comodules and contramodules}\label{sub.base}
\begin{definition}\label{def.coring}
Let $A$ be an algebra. An {\em $A$-coring} is a comonoid in the monoidal category of $A$-bimodules, i.e.\ it is a triple $\cC = (C,\Delta,\eps)$, consisting of an $A$-bimodule $C$ and $A$-bilinear maps $\Delta:C\to C\ot_A C$, $\eps: C\to A$ such that
\begin{blist}
\item $\Delta$  is coassociative, i.e.\
$$
(\Delta\ot_A\id)\circ \Delta = (\id\ot_A \Delta)\circ \Delta; 
$$
\item $\Delta$ is counital, i.e.\ both $(\eps\ot_A\id)\circ \Delta$ and $(\id\ot_A\eps)\circ \Delta$ are identities on $C$.
\end{blist}
The map $\Delta$ is called the {\em comultiplication} of $\cC$ and $\eps$ is known as the {\em counit} of $\cC$
\end{definition}

Occasionally we use the Sweedler notation
$
\Delta(c) = \sum c\sw 1\ot c\sw 2,
$
 to denote the value of $\Delta$ at $c\in C$.
 
\begin{definition}\label{def.morphism}
Let $\cC = (C,\Delta_C,\eps_C)$ be an $A$-coring an  $\cD = (D,\Delta_D,\eps_D)$ be a  $B$-coring.  A morphism from $\cC$ to $\cD$ is a pair of maps $f_0: A\to B$, $f_1: C\to D$, such that
\begin{blist}
\item $f_0$ is an algebra map and $f_1$ is an $A$-bimodule map, so that they induce the map $T(f_0,f_1): T_A(C)\to T_B(D)$ of graded algebras as explained in Example~\ref{ex.semi-free};
\item $\eps_D \circ f_1 = f_0 \circ \eps_C$;
\item $T(f_0,f_1)\circ \Delta_C = \Delta_D\circ f_1$.
\end{blist}
The category of corings is denoted by $\crg$.
\end{definition}

\begin{definition}\label{def.base.cor}
A {\em based $A$-coring} is a pair $(\cC, x)$ consisting of an $A$-coring $\cC= (C,\Delta,\eps)$ and an element $x\in C$ such that $\eps(x)=1$, called a {\em base point} . The category with objects based corings and morphisms of corings as morphisms (no condition on the base point) will be denoted by $\bcrg$.
\end{definition}

\begin{lemma}\label{lem.split} Let $\cC= (C,\Delta,\eps)$  be an $A$-coring and write $C^+ := \ker \eps$. Base points of $\cC$ are in one-to-one correspondence with the left (resp.\ right) $A$-module decompositions $C = A\oplus C^+$.
\end{lemma}
\begin{proof}
Given $x\in C$ such that $\eps(x)=1$, define the left $A$-module map
$$
\pi_x^L : C\to C^+, \qquad c\mapsto c- \eps(c)x.
$$
This map is well-defined since $\eps(x)=1$ and it clearly splits the inclusion $C^+\hookrightarrow C$. Again, since $\eps(x)=1$, the whole of $Ax$ is in the kernel of $\pi_x^L$. Conversely, if $\pi_x^L(c) = 0$, then $c = \eps(c)x \in Ax$. Therefore $\ker \pi^L_x = Ax \cong A$.

For the right $A$-module decomposition the surjection $C\to C^+$ is given by
$$
\pi_x^R : C\to C^+, \qquad c\mapsto c- x\eps(c).
$$

Conversely, given a left $A$-module decomposition $C = A\oplus C^+$, define $x$ as the image of $1\in A$ under the section $A\to C$ of $\eps$. 
\end{proof}

\begin{lemma}\label{base.sur}
The category $\bcrg$ of based corings is equivalent to the full subcategory $\scrg$ of $\crg$ consisting of corings with a surjective counit.
\end{lemma}
\begin{proof}
The surjectivity of the counit in a coring is equivalent to the existence of a base point: If $\eps$ is surjective, then any $x\in \eps^{-1}(1)$ is a base point; if a base point exists, then surjectivity of $\eps$ follows by the direct sum decompositions in Lemma~\ref{lem.split}. Since the morphisms in $\bcrg$ do not fix base points, the forgeful functor 
$$
\bcrg \to \scrg, \qquad (\cC, x) \mapsto \cC,
$$
provides the required equivalence of categories.
\end{proof}

A representation of a coring is typically encoded in the familiar notion of a comodule.

\begin{definition}\label{def.comodule}
A right {\em comodule} over an $A$-coring $\cC= (C,\Delta,\eps)$ is a pair $(M,\roM)$, where $M$ is a right $A$-module with $A$-multiplication $\mu: M\ot A\to M$, and $\roM: M\to M\ot_A C$ is a right $A$-linear map such that  $(\roM\ot_A \id - \id\ot_A\Delta)\circ\roM$ vanishes on $M$, while $\mu\circ (\id\ot \eps)\circ \roM$ is the identity on $M$. 

A {\em morphism of $\cC$-comodules} $\varphi:  (M,\roM) \to (N,\roN)$ is a right $A$-linear map $\varphi: M\to N$, rendering commutative the following diagram:
$$
\xymatrix{ M \ar[rr]^-{\varphi} \ar[d]_{\roM} && N\ar[d]^{\roN}\\
M\ot_A C \ar[rr]^-{\varphi \ot \id} && N\ot_A C.}
$$
\end{definition}

Less familiar is the notion of a contramodule, which, after years of neglect, has been brought to the fore in a recent monograph \cite{Pos:hom}. Recall that given a right $A$-module $L$, any map of right $A$-modules $f: M\to N$ induces two maps between the hom-sets:
$$
\rhom A L f : \rhom ALM\to \rhom ALN, \qquad g \mapsto f\circ g,
$$
 and
 $$
\rhom A f L: \rhom AML\to \rhom ANL, \qquad g \mapsto g\circ f.
$$
Furthermore,  if $L$ is an $A$-bimodule, then $\rhom A L M$ is a right $A$-module by $(fa)(u) := f(au)$, for all $a\in A$, $u\in L$ and $f\in \rhom A L M$, and there is an adjunction isomorphism:  $\rhom A M {\rhom A L N} \cong \rhom A {M\ot_A L}  N$, $f\mapsto [m\ot l\mapsto f(m)(l)]$. Finally, note that $\rhom AAM \cong M$ by the evaluation at $1\in A$. With all these preliminaries at hand one can state the following definition.

\begin{definition}\label{def.contramodule} 
A right {\em contramodule} over an $A$-coring $\cC= (C,\Delta,\eps)$ is a pair $(M,\alpha)$, where $M$ is a right $A$-module and $\alpha: \rhom A C M\to M$ is a right $A$-linear map such that  $\alpha\circ (\rhom A C \alpha  - \rhom A\Delta M)$ vanishes on $\rhom A {C\ot_A C} M\cong \rhom A C {\rhom AC M}$, while $\alpha \circ \rhom A\eps M $ is the identity on $\rhom AAM \cong M$. 

A {\em morphism of $\cC$-contramodules} $\varphi:  (M,\alpha_M) \to (N,\alpha_N)$ is a right $A$-linear map $\varphi: M\to N$, rendering commutative the following diagram:
$$
\xymatrix{ \rhom A C M \ar[rr]^-{\rhom A C \varphi} \ar[d]_{\alpha_M} && \rhom A CN\ar[d]^{\alpha_N}\\
M \ar[rr]^-{\varphi} && N.}
$$

\end{definition}

Further information about corings can be found in \cite{BrzWis:cor}.

\section{Based corings and curved differential graded algebras}
\setcounter{equation}{0}
In this section an equivalence between the categories of semi-free curved differential algebras and corings with surjective counits is established.
\subsection{From corings to curved differential graded algebras...}
\begin{theorem}\label{thm.crg-cdga}
\begin{zlist}
\item Given an $A$-coring $\cC=(C,\Delta,\eps)$ and  $x\in C$,  define
\begin{blist}
\item for all $a\in A$, $d_x(a) = [x,a]$, and for all $c\in C^{\otimes_A n}$,
$$
d_x (c) = x\ot c + \sum_{k=1}^n(-1)^k \Delta_k^{(n)}(c) + (-1)^{n+1}c\ot x,
$$
where $\Delta_k^{(n)} = C^{\otimes_A k-1}\ot \Delta\ot C^{\otimes n-k} : C^{\otimes_A n} \to C^{\otimes_A n+1};
$
\item $\gamma_x = x\ot x - \Delta(x)$.
\end{blist}
Then $T^\flat(\cC,x) := (T_A(C), d_x, \gamma_x)$ is a curved differential graded algebra.
\item Let $\cC$ be an $A$-coring, let $\cD$ be a $B$-coring,  and let $x\in C$ and $y\in D$.   Given a morphism $(f_0, f_1)$ from $\cC$ to  $\cD$,
the pair $(T(f_0, f_1), f_1(x)-y)$ is a morphism $T^\flat(\cC,x) \to T^\flat(\cD,y)$ of curved differential graded algebras.
\item For a based $A$-coring $(\cC, x)$ and a based $B$-coring $(\cD,y)$, the assignment
$$
T : (\cC,x)\mapsto  (T_A(C^+), d_x, \gamma_x), \qquad (f_0, f_1)\mapsto (T(f_0, f_1)\mid_{T_A(C^+)}, f_1(x)-y),
$$ 
described in (1) and (2), defines a  functor from $\bcrg$ to $\cdga$. 
\end{zlist}
\end{theorem}
\begin{proof}
(1) Take any $c\in C^{\otimes_A m}$ and  $c'\in C^{\otimes_A n}$. Then
\begin{eqnarray*}
d_x(c\ot c') &=& x\ot c\ot c' + \sum_{k=1}^{m+n}(-1)^k \Delta_k^{(m+n)}(c\ot c') + (-1)^{m+n+1}c\ot c'\ot x\\
&=& x\ot c\ot c' + \sum_{k=1}^{m}(-1)^k \Delta_k^{(m)}(c)\ot c' + (-1)^{m+1}c\ot  x \ot c'\\
&&+ (-1)^{m}(c\ot  x \ot c' + \sum_{k=1}^{n}(-1)^k c\ot \Delta_k^{(n)}(c') + (-1)^{n+1}c\ot c'\ot x)\\
&=& d_x(c)\ot c' + (-1)^m c\ot d_x(c'),
\end{eqnarray*}
so $d_x$ is a graded derivation as required. Applying $d_x$ twice to $c\in  C^{\otimes_A n}$ one easily finds that due to the sign alteration and the coassociativity of $\Delta$ all terms involving $\Delta_k^{(n)}(c)$ cancel out and one is only left with the first two and the last two terms, so that
$$
d_x\circ d_x (c) = x\ot x\ot c - \Delta(x)\ot c + c\ot \Delta(x) - c\ot x\ot x = [\gamma_x, c],
$$
as required. The Bianchi identity follows by the coassociativity, since
\begin{eqnarray*}
d_x(\gamma_x) &=& x\ot x\ot x - \Delta(x)\ot x + x\ot \Delta(x) - x\ot x\ot x\\
&& - x\ot \Delta(x) + (\Delta\ot \id)\circ \Delta(x) - (\id\ot \Delta)\circ \Delta(x) + \Delta(x)\ot x = 0.
\end{eqnarray*}
This completes the proof that $T^\flat(\cC,x)$ is a curved differential graded algebra.

(2) As explained in Example~\ref{ex.semi-free}, $T(f_0,f_1)$ is a morphism of graded algebras $T_A(C)\to T_B(D)$, so only conditions (a) and (b) in Definition~\ref{def.cat} need be checked. Set $\omega = f_1(x) -y$. Then, for all $a\in A$,
\begin{eqnarray*}
T(f_0,f_1)(d_x a) &=& f_1([x,a]) = [f_1(x),f_0(a)] \\
&=& [y,f_0(a)] + [\omega, f_0(a)] = 
d_yT(f_0,f_1)(a) + [\omega, T(f_0,f_1)(a)].
\end{eqnarray*}
Furthermore, for all $c\in C^{\ot_A  n}$,
\begin{eqnarray*}
T(f_0,f_1)(d_x c) &=& T(f_0,f_1)(x\ot c) + \sum_{k=1}^n(-1)^k T(f_0,f_1)(\Delta_k^{(n)}(c))\\
&& + (-1)^{n+1}T(f_0,f_1)(c\ot x)\\
&=& d_y T(f_0,f_1)(c) + \omega \ot_B T(f_0,f_1)(c) +(-1)^{n+1}T(f_0,f_1)(c)\ot_B \omega\\
&=& d_y T(f_0,f_1)(c) + [\omega , T(f_0,f_1)(c)],
\end{eqnarray*}
where the second equality follows by the fact that $(f_0,f_1)$ is a morphism of corings.
Finally,
\begin{eqnarray*}
T(f_0,f_1)(\gamma_x) &=& f_1(x)\ot f_1(x) - \sum f_1(x\sw 1)\ot f_1(x\sw 2)\\
& =& y\ot y +y\ot  \omega + \omega\ot  y +\omega\ot  \omega - \sum f_1(x)\sw 1\ot  f_1(x)\sw 2 \\
&=& y\ot y +y\ot \omega + \omega\ot y +\omega\ot \omega - \sum y\sw 1\ot y\sw 2 - \sum \omega\sw 1\ot \omega \sw 2\\
&=& \gamma_y +d_y\omega +\omega^2,
\end{eqnarray*}
where the second equality follows by the fact that $(f_0,f_1)$ is a morphism of corings. Therefore, $(T(f_0,f_1), f_1(x) -y)$ is a morphism of curved differential graded algebras as claimed.

(3) Since $\eps(x)=1$, $d_x(a) \in C^+$, for all $a\in A$. An easy calculation that uses properties of the counit and the definition of maps $\pi_x^L$ and $\pi_x^R$ in the proof of Lemma~\ref{lem.split}, yields the following formula, for all $c= c^1\ot_A\ldots c_n\in {C^+}^{\ot_A n}$,
$$
d_x (c) =\sum_{k=1}^n(-1)^k c^1\ot \ldots c^{k-1}\ot \left((\pi_x^R\ot \pi_x^L)\circ\Delta(c^{k})\right)\ot c^{k+1}\ot \ldots \ot c^n.
$$
This means that $d_x$ restricts to $T_A(C^+)$. By the same token, $\gamma_x = x\ot x - \Delta(x)\in C^+\ot_A C^+$. Hence $T(\cC,x)$ is a curved differential graded algebra. Finally, since for a morphism of corings $\eps_D\circ f_1 = f_0\circ \eps_C$ and $\eps_D(y) =1$, $\eps_C(x) =1$, we conclude that $f_1(x) -y\in D^+$.

Statements (1) and (2) establish that $T$ is a mapping from $\bcrg$ to $\cdga$. To check whether $T$ preserves the composition, take morphisms $(f_0,f_1): (\cC,x)\to  (\cD,y)$ and $(g_0,g_1): (\cD,y)\to  (\cE,z)$ and compute
\begin{eqnarray*}
T\left((g_0,g_1)\circ (f_0,f_1)\right) &=& (T(g_0\circ f_0,g_1\circ f_1), g_1\circ f_1(x) -z)\\
&=& (T(g_0,g_1)\circ T(f_0,f_1), g_1( f_1(x)-y) + g_1(y) -z),
\end{eqnarray*}
as required.
\end{proof}

\begin{corollary}\label{cor.trivial} For any choice of $x, y\in C$,  $T^\flat(\cC,x)$ and $T^\flat(\cC,y)$ are isomorphic as curved  differential graded algebras. Consequently, for any $x$, $T^\flat(\cC,x)$ is isomorphic to a differential graded algebra.
\end{corollary}
\begin{proof} 
The identity isomorphism of corings $\cC\to \cC$ induces the required isomorphism
$$
(\id, x-y) : T^\flat(\cC,x)\to T^\flat(\cC,y), 
$$
of curved differential graded algebras. If $y=0$, then  the curvature  element $\gamma_y$ is zero, hence $T^\flat(\cC,y)$ is an ordinary differential graded algebra (such that $A\subseteq \ker d_0$), and any $T^\flat(\cC,x)$ is thus isomorphic to a differential graded algebra. 
\end{proof}

\begin{corollary}\label{cor.group} For any choice of base points $x, y\in C$,   $T(\cC,x)$ and $T(\cC,y)$ are isomorphic as curved differential graded algebras. In particular, if $\cC$ has a group-like element, then these are isomorphic (as curved differential graded algebras) to  differential graded algebras.
\end{corollary}
\begin{proof} 
The  isomorphism  described in the proof of Corollary~\ref{cor.trivial} restricts to the isomorphism of curved differential graded algebras  $T(\cC,x)\to T(\cC,y)$.  In particular, if, say, $y$ is a group-like element in $\cC$, then $y$ is a base point and $\Delta(y) = y\ot y$, hence $\gamma_y =0$ and  $T(\cC,y)$ is a differential graded algebra. 
\end{proof}

\begin{proposition}\label{prop.comod}
Let $\cC = (C,\Delta, \eps)$ be an $A$-coring and let $x\in C$.  Given a complex $\delta^l :(M^l,\varrho_l)\to (M^{l+1},\varrho_{l+1})$ of  $\cC$-comodules, define
$$
\cov^{n,l} : M^l\to M^{l-n+1}\ot_A C^{\otimes_A n}, \qquad 
\cov^{n,l} = \begin{cases} \delta^l & \mbox{for $n =0$} \\
m\mapsto (-1)^l \varrho_l(m) - (-1)^lm\ot x, & \mbox{for $n =1$}\\
0  & \mbox{for $n\geq 2$}.
\end{cases}
$$ 
Then $(M^\bullet, \cov)$ is an integrable $\ZZ$-connection for the curved differential graded algebra $T^\flat(\cC,x)$. Furthermore, if $\eps(x)=1$, then $( M^\bullet, \cov)$ is an integrable $\ZZ$-connection for the curved differential graded algebra $T(\cC,x)$ associated to the based $A$-coring $(\cC,x)$.
\end{proposition}
\begin{proof}
For any $a\in A$ and $m\in M^l$,
\begin{eqnarray*}
\cov^{1,l}(ma) &=& (-1)^l\varrho_l(ma) - (-1)^lma\ot x\\
& =& (-1)^l\varrho_l(m)a - (-1)^lm\ot x a + (-1)^lm\ot [x,a]\\
& = & \cov^{1,l}(m)a + (-1)^lm\ot d_x(a),
\end{eqnarray*}
hence $(-1)^l\cov^1$ satisfies the graded Leibniz rule on $M^l$. Since the $\delta^l$ are right $A$-linear, the sum $\sum_n \cov ^n = \cov^0 + \cov ^1$  extends to a covariant derivative on $M^\bullet\ot_A T_A(C)^\bullet$; see Example~\ref{ex.connection}. Explicitly, in view of the definition of $d_x$ in Theorem~\ref{thm.crg-cdga}(1)(a) one obtains, for all $c\in C^{\ot_An}$ and $m\in M^l$,
$$
\cov(m\ot c) = \delta^l(m)\ot c + (-1)^l  \varrho_l(m)\ot c + \sum_{k=1}^n(-1)^{k+l} m\ot \Delta_k^{(n)}(c) - (-1)^{l+n}m\ot c\ot x.
$$
Applying $\cov$ to the above formula, and taking into account the coassociativity of the $\varrho_l$ and $\Delta$  as well as the alternating signs, one finds that all the terms involving $\varrho_l(m)$ and $\Delta_k^{(n)}(c) $ cancel out, except for
\begin{eqnarray*}
\cov\circ\cov (m\ot c) &=& \delta^{l+1}\circ \delta^l(m) \ot c + (-1)^l\left(\left(\delta^l\ot \id\right)\circ \varrho_l -  \varrho_{l+1} \circ \delta^l\right)(m)\ot c \\
&&+ m\ot c\ot \Delta(x) - m\ot c \ot x\ot x.
\end{eqnarray*}
The first term in the above formula vanishes, since the sequence $\delta^l$ forms a complex, the second term vanishes by the colinearity of the $\delta^l$. Consequently, one is left with $\cov\circ\cov (m\ot c)= -m\ot c \ot \gamma_x$. Therefore, $(M, \cov)$ is an integrable $\ZZ$-connection for the curved differential graded algebra $T^\flat(\cC,x)$. 

Clearly, if $x$ is a base point, then $\cov^{1,l} (M^l)\subseteq M^l\ot_A \cC^+$, so, by the same arguments as before, $(M^\bullet, \cov)$ is an integrable $\ZZ$-connection for the curved differential graded algebra $T(\cC,x)$.
\end{proof}

\begin{example}\label{ex.entw}
Recall from \cite{BrzMaj:coa} that an {\em entwining structure} is a triple $(A,C,\psi)$ consisting of an  algebra $A$ (with product $\mu: A\ot A\to A$ and unit $\iota: \KK\to A$), a coalgebra $C$ (with comultiplication $\Delta: C\to C\ot C$ and counit $\eps: C\to \KK$) and a $\KK$-linear map $\psi: C\ot A\to A\ot C$ making the following {\em bow-tie} diagram commute
$$
\xymatrix{
& C\ot  A\ot  A \ar[ddl]_{\psi\otimes \id} 
\ar[dr]^{\id\otimes \mu}  &  &             
C\ot  C\ot  A \ar[ddr]^{\id\otimes\psi}& \\
& & C\ot  A \ar[ur]^{\Delta\otimes \id} 
\ar[dr]^{\eps\ot \id}\ar[dd]^{\psi}  & &\\
A\ot C\ot  A \ar[ddr]_{\id\otimes\psi} & C \ar[ur]^{\id\otimes 
\iota} \ar[dr]_{\iota\otimes \id}& &A  & C\ot  A\ot  C 
\ar[ddl]^{\psi\otimes \id}\\
& &  A\ot  C \ar[ur]_{\id\otimes\eps} \ar[dr]_{\id\otimes\Delta} & &\\
& A\ot  A\ot  C \ar[ur]_{\mu\otimes \id} &  & A\ot C\ot  C & . }
$$
As explained in \cite{Brz:str} to an entwining structure one associates an $A$-coring $\cC(A,C,\psi) = (A\ot C, \Delta_{A\ot C}, \eps_{A\ot C})$, where  the $A$-multiplications on  $A\ot C$ are given by
$$
a(a'\ot c) a'' := aa'\psi(c\ot a''),
$$
and with structure maps
$
\Delta_{A\ot C} = \id\ot \Delta $ and $ \eps_{A\ot C} = \id\ot \eps$. If $A$ is flat as a $\KK$-module, then $\ker \eps_{A\ot C} = A\ot C^+$, where $C^+$ stands for the kernel of $\eps$. In this case $T_A((A\ot C)^+)^n = A\ot {C^+}^{\ot n}$. Take any $e\in C$ such that $\eps(e)=1$ and $x=1\ot e$. Then, for all $a\in A$, $c\in C^+$,
$$
d_x(a) = \psi(e\ot a) - a\ot e, \quad d_x(a\ot c) = \psi(e\ot a) \ot c - a\ot \Delta(c) + a\ot c\ot e.
$$

If $(V,\roV)$ is a right $C$-comodule, then $V\ot A$ is a $\cC(A,C,\psi)$-comodule with the obvious right $A$-multiplication and with the coaction given as the composite
$$
\xymatrix{ \varrho_0 : V\ot A \ar[rr]^-{\roV\ot \id} && V\ot C\ot A \ar[rr]^-{\id \ot \psi} && V\ot A\ot C \cong V\ot A \ot_A( A\ot C)}.
$$
In particular $V\ot A$ is a $C$-comodule. Iterating this procedure one obtains a family of $\cC(A,C,\psi)$-comodules $(M^l := V\ot A^{\ot l+1}, \varrho_l)$, $l\in \NN$. For each $l\in \NN$ and $k = 0, \ldots , l$, define right $A$-module maps
$$
\iota^{(l)}_k:  A^{\ot l+1} \to A^{\ot l+2}, \qquad  \iota^{(l)}_k = {A^{\ot k}}\ot \iota \ot  {A^{\ot l+1-k}}.
$$
These give rise to a complex of $A$-modules
$$
\delta^l: V\ot A^{\ot l+1}\to V\ot A^{\ot l+2}, \qquad \delta^l = \sum_{k=0}^l (-1)^k V\ot \iota^{(l)}_k.
$$
Since the unit map $\iota: \KK\to A$ is preserved by $\psi$ (this is the meaning of the left triangle in the bow-tie diagram), all $\delta^l$ are right $\cC(A,C,\psi)$-comodule maps, and we obtain integrable $\ZZ$-connections as described in Proposition~\ref{prop.comod}.
\end{example}

\begin{example}\label{ex.comatrix}
Let $A$ be an algebra. To any finitely generated projective right $A$-module $P$ one can associate a coring, known as a {\em comatrix coring} \cite{ElKGom:com}: Take any subalgebra $B$ of the endomorphism ring $S = \rend A P$. Then $P$ is a $(B,A)$-bimodule with the $B$-multiplication $b p = b(p)$, for all $b\in B$, $p\in P$. Let $P^* := \rhom A PA$ be the dual $(A,B)$-bimodule. Since $P$ is a fintely generated projective module the map 
$$
\Theta: P\ot _A P^* \to S, \qquad p\ot \chi \mapsto [q\mapsto p\chi(q)],
$$
is an $S$-bimodule isomorphism. Denote by $e = \Theta^{-1}(1_S)$ the dual basis element. Then $\cC = (C,\Delta, \eps)$, where $C= P^*\ot_B P$, 
$$
\Delta : \chi \ot p\mapsto \chi \ot e \ot p, \qquad \eps: \chi\ot p \mapsto \chi(p),
$$
is an $A$-coring. Since the counit is the evaluation map, $\eps$ is surjective if and only if the right $A$-module $P$ is a generator (and since it is also finitely generated and projective, $P$ is a progenerator). In this case, for any $x = \sum_i \xi_i \ot x_i\in \eps^{-1}(1_A) \subseteq P^*\ot_B P$, the curved differential graded algebra $T(\cC, x)$ can be described as follows. 

First, in view of the isomorphism $\Theta$, the tensor algebra $T_A(C)$ is
$$
T_A(C)^0 = A  \quad  \mbox{and} \quad T_A(C)^n = P^*\ot_B S^{\ot _B n-1} \ot_B P, \quad \mbox{for $n\geq 1$}, 
$$
with the product $(\chi \ot s \ot p)(\chi'\ot s'\ot p') = \chi \ot s \ot \Theta(p\ot \chi') \ot s' \ot p'$. For all $n\in \NN$ and $i = 0, \ldots ,n$,
define $\Phi_0^{(0)} = \eps : P^*\ot_B P\to A$, and, for a positive $n$, 
$$
\Phi_i^{(n)} : P^*\ot_B S^{\ot _B n} \ot_B P \to P^*\ot_B S^{\ot _B n-1} \ot_B P, 
$$
by
$$
\Phi_i^{(n)} : \chi \ot s_1 \ot \ldots \ot s_{n} \ot p \mapsto 
\begin{cases} \chi \circ s_1 \ot s_2 \ot \ldots \ot s_{n} \ot p, & \mbox{$i=0$,}\\
\chi \ot s_1 \ot \ldots \ot s_is_{i+1}\ldots \ot s_{n} \ot p, & \mbox{$i=1, \ldots , n-1$,}\\
\chi \ot s_1 \ot \ldots \ot s_{n-1} \ot s_{n}(p), & \mbox{$i= n$.}
\end{cases}
$$
Then 
$$
T_A(C^+)^0 = A \quad  \mbox{and} \quad T_A(C^+)^n = \bigcap_{i=0}^{n-1} \ker \Phi^{(n-1)}_i, \quad \mbox{for $n\geq 1$}.
$$
The derivative comes out as $d_x(a) = [x,a]$ on $A$,  
$$
d_x(\sum_j\chi_j \ot p_j) = \sum_{i,j} \xi_i\ot \Theta(x_i \ot \chi_j)\ot p_j - 
\sum_j\chi_j \ot 1_S\ot p_j + \sum_{i,j} \chi_j\ot  \Theta(p_j \ot\xi_i) \ot x_i,
$$
on $C^+$, and then is extended to the whole of $T_A(C^+)$ by the Leibniz rule. The curvature element is 
$$
\gamma_ x = \sum_{i,j} \xi_i\ot \Theta(x_i \ot \xi_j)\ot x_j - \sum_i\xi_i \ot 1_S\ot x_i.
$$ 
We denote this curved differential graded algebra  by $\cA({}_BP_A, x)$.

$P$ is a right $\cC$-comodule with the coaction
$p\mapsto e\ot_B p$, hence it is an integrable $\ZZ$-connection over $T(\cC, x)$ with covariant derivative:
$$
\cov^1: p\mapsto e\ot_B p - p\ot_A x.
$$
Since this connection is concentrated in degree zero (hence bounded from both directions) and $P$ is a projective $A$-module, $(P, \cov)$ is a cohesive module over  $\cA({}_BP_A, x)$.
\end{example}

\begin{proposition}\label{prop.pre-Galois}
Let $\cA = (A^\bullet,d_A,\gamma_A)$ be a curved differential graded algebra and let $\Pp = (P,\cov)$ be a cohesive curved $\cA$-module, concentrated in degree 0 and such that $P_A$ is a progenerator. Denote by $S$ the degree zero part of the endomorphism ring $\rend \cA \Pp$  and take $B\subseteq S$ such that $d_S(B) =0$ (i.e.\ $B$ is a subalgebra of the endomorphism ring  of $\Pp$ in the homotopy category of $\cA$-modules). Write $\overline{P^*\ot_B P}$ for the kernel of the evaluation map $\eps: P^*\ot_B P\to A$, take any $x \in \eps^{-1}(1)$ and define   $\theta_1 : \overline{P^*\ot_B P}\to A^1$ as the restriction of the left $A$-linear map 
$$
\theta^L : P^*\ot_B P \to A^1, \qquad \theta^L= (\eps\ot \id)\circ ( \id \ot \cov),
$$
Then the pair $\theta_0 = \id: A\to A$ and $\theta_1$together with $\omega = \theta^L(x)$, defines a morphism of curved differential graded algebras from $\cA({}_BP_A,x)$ to $\cA$.
\end{proposition}
\begin{proof}
By construction, $\theta_1$ is a left $A$-linear map. Define
$$
\theta^R : P^*\ot_B P \to A^1, \qquad \theta^R = \theta^L - d_A\circ \eps.
$$
Then $\theta^R$ is right $A$-linear  by the following simple calculation that uses the Leibniz rule for $\cov$. First,  for all $p\in P$, write  $\cov(p) : = \sum p\suc 0\ot p\suc 1$, and then compute, for all $\chi\in P^*$, $a\in A$, 
\begin{eqnarray*}
\theta^R(\chi\ot pa) &=& \sum \chi((pa)\suc 0)(pa)\suc 1 - d_A(\chi(pa)) \\
&=& \sum \chi(p\suc 0)p\suc 1 a +\chi(p)d_Aa - d_A(\chi(p)) a - \chi(p)d_Aa = \theta^R(\chi\ot p)a.
\end{eqnarray*}
Note that $\theta_1$ is the restriction of $\theta^R$ to $\ker\eps = \overline{P^*\ot_B P}$, hence $\theta_1$ is $A$-bilinear and thus, together with $\theta_0 =\id$ defines an algebra map $T_A(\overline{P^*\ot_B P}) \to A^\bullet$.

Write $x = \sum_i \xi_i \ot x_i$, so that $\omega = \sum_i \xi_i(x_i\suc 0)x_i\suc 1$. Owing to the definition of the derivative $d_x$, the Leibniz rule for a covariant derivative and the fact that $\sum_i\xi_i(x_i) =1$ one obtains,
\begin{eqnarray*}
\theta_1 (d_xa) &=& \theta_1 (xa - ax) = \sum_i \xi_i((x_ia)\suc 0)(x_ia)\suc 1 - a \sum_i \xi_i(x_i\suc 0)x_i\suc 1\\
&=& \sum_i \xi_i(x_i\suc 0)x_i\suc 1a + \sum_i \xi_i(x_i)d_Aa - a \sum_i \xi_i(x_i\suc 0)x_i\suc 1 = d_Aa + [\omega, a],
\end{eqnarray*}
as required. Finally, in terms of the explicit notation for the covariant derivative, the integrability condition takes the form
\begin{equation}\label{integrability}
\sum p\suc 0 \suc 0 \ot p\suc 0\suc 1 p\suc 1 + \sum p\suc 0 \ot d_Ap\suc 1 + p\ot\gamma_A =0.
\end{equation}
Writing $e= \sum_j e_j\ot \varphi_j \in P\ot_A P^*$ for the dual basis element, using the definition of $\gamma_x$ in Example~\ref{ex.comatrix}, and observing that $\omega = \theta^L(x) = \theta^R(x)$ (for $\sum_i\xi_i(x_i) =1$), one can compute
\begin{eqnarray*}
\theta_2(\gamma_x) &=& \theta^R(x)\theta^L(x) - \sum_{i,j} \theta^R(\xi_i\ot e_j)\theta^L(\varphi_j\ot x_i)\\
&=& \omega^2 - \sum_{i,j}\xi_i(e_j\suc 0)e_j\suc 1\varphi_j(x_i \suc 0) x_i\suc 1 + \sum_{i,j}d_A(\xi_i(e_j))\varphi_j(x_i \suc 0) x_i\suc 1\\
&=& \omega^2  - \sum_{i,j}\xi_i((e_j\suc 0\varphi_j(x_i \suc 0))\suc 0)(e_j\suc 0\varphi_j(x_i \suc 0))\suc 1 x_i\suc 1  \\
&&+ \sum_{i,j}d_A(\xi_i(e_j))\varphi_j(x_i \suc 0) x_i\suc 1 + \sum_{i,j}\xi_i(e_j)d_A(\varphi_j(x_i \suc 0)) x_i\suc 1\\
&=& \omega^2 - \sum_i \xi_i(x_i\suc 0 \suc 0)x_i\suc 0\suc 1 x_i\suc 1 + \sum_i d_A(\xi_i(x_i\suc 0)) x_i\suc 1\\
&=& \omega^2 + \sum_i \xi_i(x_i)\gamma_A +\sum_i \xi_i(x_i\suc 0)d_A x_i\suc 1 + d_A(\xi_i(x_i\suc 0)) x_i\suc 1 = \omega^2 +\gamma_A + d_A\omega.
\end{eqnarray*}
The first equality follows by the definition of $\gamma_x$, the third one is a consequence of the Leibniz rule for $\cov$. The fourth equality is obtained by a combination of the Leibniz rule for $d_A$ and the dual basis property. Next, Equation~\eqref{integrability} is used. Finally we employed the definition of $\omega$ and the equality $\sum_i\xi_i(x_i) =1$. This completes that proof that the triple $(\theta_0, \theta_1, \omega)$ defines a morphism of curved differential graded algebras.
\end{proof}
\begin{example}\label{ex.order}
Let $S$ be a set and $Q\subset S\times S$ be a relation that is reflexive, transitive and locally finite in the sense that, for all $s,t\in S$, the set
$$
\omega(s,t) = \{u\in S \; |\; (s,u)\in Q\;\; \mbox{and}\;\; (u,t)\in Q\}
$$
is finite. Take any algebra $A$, consider free left $A$-module $C=A^{(Q)}$ and make it into an $A$-bimodule by
$$
\left(\sum_i a_i (s_i,t_i)\right)b = \sum_i a_i b(s_i,t_i).
$$
Define $A$-bimodule maps $\Delta: C\to C\ot_A C$, $\eps: C\to A$ by
$$
\Delta: (s,t) \mapsto \sum_{u\in \omega(s,t)} (s,u)\ot (u,t), \qquad \eps: (s,t)\mapsto \delta_{st}.
$$
Then $(C,\Delta, \eps)$ is an $A$-coring; see \cite[Section~17.4]{BrzWis:cor}. The kernel of the counit consists of all $\sum_ia_i(s_i,t_i)$ such that $\sum_i a_i \delta_{s_it_i} =0$, hence $T_A(C^+)^n$ is a submodule of $A^{(Q^n)}$ consisting of all $\sum_ia_i(s^1_i,t^1_i, \ldots  ,s^n_i,t^n_i )$ such that $\sum_i a_i \delta_{s^k_it^k_i} =0$, for all $k=1,2,\ldots , n$. The multiplication in $T_A(C^+)$ is by concatenation. Fix any $e\in S$, then $(e,e) \in C$ and $\eps(e,e)=1$, so we can construct the curved differential graded algebra $T(\cC, (e,e))$. The differential acts trivially on $A$, while on $C^+$ it is a restriction of the map 
$$
d(s,t) = (e,e,s,t) - \sum_{u\in \omega(s,t)} (s,u,u,t) + (s,t,e,e), \qquad \mbox{for all}\;\; (s,t)\in Q.
$$
The curvature element is 
$$
\gamma = - \sum_{u\in \omega(e,e),\;\; u\neq e} (e,u,u,e).
$$
\end{example}

\begin{example} \label{ex.matrix}
A curved differential graded algebra built on traceless matrices is a special case of both Example~\ref{ex.comatrix} and Example~\ref{ex.order}. Let $C = M_N(A)$ be an $A$-bimodule of $N\times N$-matrices with entries from an algebra $A$. Then one can define comultiplication and counit on $C$, by setting,
$$
\Delta(E_{ij}) = \sum_{k=1}^NE_{ik}\ot E_{kj}, \qquad \eps(E_{ij}) = \delta_{ij},
$$ 
where $E_{ij}$ are matrix units which generate $M_N(A)$ as a left (or right) $A$-module. This is an example of a comatrix coring (with $P=A^N$) or a coring associated to a partial order (with $S=\{1,\ldots , N\}$ and $Q = S\times S$). The kernel of the counit is equal to the submodule of all traceless $N\times N$-matrices in $C$. This is generated by $E_{ij}$, $i\neq j$ and $F_i := E_{NN} - E_{ii}$, $i=1,\ldots , N-1$.  The differential of $T(\cC, E_{NN})$ acts trivially on $A$, and
$$
dE_{ij} = F_i\ot E_{ij} - \sum_{k\neq i,j} E_{ik}\ot E_{kj} + E_{ij}\ot F_j, \quad
dF_{i} = F_i\ot F_{i} +\sum_{k\neq i} (-1)^{\delta_{Ni}}E_{ik}\ot E_{ki} .
$$
The curvature element is $-\sum_{k=1}^{N-1}E_{Nk}\ot E_{kN}$.
\end{example}

Any morphism $f: M\to N$ of right $\cC$-comodules $(M,\roM)$, $(N,\roN)$ induces a right $T_A(C)$-module map $T_A(f,C):= f\ot\id: M\ot_AT_A(C)\to N\ot_AT_A(C)$, i.e.\ a morphism of $T^\flat(\cC, x)$-modules $\cM = (M\ot_AT_A(C)^\bullet, \cov_M)$ to $\cN = (N\ot _AT_A(C)^\bullet, \cov_N)$. However, the compatibility of $f$ with coactions means that
$d_{\cM,\cN}(T_A(f,C)) =0$, i.e.\ comodule maps give rise to elements of the zero cohomology of the complex 
$$\rhom{T^\flat(\cC, x)}{\left(M\ot_AT_A(C)^\bullet, \cov_M\right)}{\left(N\ot_AT_A( C)^\bullet, \cov_N\right)}.$$
Similar remarks apply to $(M\ot_AT_A(C^+)^\bullet, \cov)$  in the case of a based coring $(\cC,x)$. In this way the category of comodules can be seen as a  subcategory of both $Z^0(\mod$-$T(\cC,x))$ and the homotopy category of $\mod$-$T(\cC,x)$.

The contents of Proposition~\ref{prop.comod} can be explored further to introduce the differential graded category of complexes of comodules and to reveal the triangulated structure of its homotopy category. 

\begin{definition}\label{def.comp.com}
Let $\cC = (C,\Delta, \eps)$ be an $A$-coring. An object of the category of {\em $\cC$-comodule complexes} or, taking a cue from \cite{Blo:dua},  {\em quasi-cohesive comodules}, is the family of triples $\cM:= (M^l, \varrho^M_l, \delta^l_M)$, $l\in \ZZ$, in which each $M^l$ is a right $\cC$-comodule with coaction $\varrho^M_l: M^l\to M^l\ot_A C$ and $\delta_M^l: M^l\to M^{l+1}$ are $\cC$-comodule maps such that $\delta_M^{l+1}\circ\delta_M^l =0$.  A morphism $\varphi: \cM\to \cN$ is a family of right $A$-module maps
$$
\varphi^{k,s}_l : M^l \to N^{l+s-k} \ot_A C^{\ot_A k}, \qquad k\in \NN, \; s,l\in \ZZ,
$$ 
such that for all $l\in \ZZ$, $k\in \NN$ and $m\in M^l$, $\varphi^{k,s}_l (m) \neq 0$ for a finite number of values of $s$. The composition of $\varphi: \cM\to \cN$ with $\psi: \cN\to \Pp$ is defined by
\begin{equation}\label{comp}
(\psi\circ\varphi)^{k,t}_l = \sum_{i,s} (\psi^{k-i, t-s}_{l+s-i}\ot {C^{\ot_A i}})\circ \varphi^{i,s}_l.
\end{equation}
The category of $\cC$-comodule complexes is denoted by \com\cC, and we write $\Rhom \cC \cM\cN$ for the sets of morphisms.
\end{definition}

The category \com\cC~ is a differential graded category. The degree $s$ elements in $\Rhom \cC \cM\cN$  are all those morphisms $\varphi$ with components $\varphi^{k,s}_l$ (fixed value of $s$). Once $s$ is specified, we write $\varphi_{l}^k$ for $\varphi^{k,s}_l$. For any right $A$-module map $\phi: M^l\to N^{l+s-k}\ot_A C^{\ot_A {k}}$ define the right $A$-module map $b(\phi):M^l\to N^{l+s-k}\ot_A C^{\ot_A {k+1}}$ by
$$
b(\phi) = \left(\varrho^N_{l+s-k}\ot C^{\ot_A {k}}\right)\circ \phi+ \sum_{i=1}^{k}(-1)^i\left(N^{l+s-k}\ot \Delta_i^{(k)}\right)\circ \phi +(-1)^{k+1}(\phi \ot C)\circ \varrho^{M}_l.
$$
Then, for all degree $s$ morphisms $\varphi: \cM\to\cN$, the differential is defined by
\begin{equation}\label{d.phi}
(d\varphi)^k_l = (\delta_N^{l+s-k}\ot  {C^{\ot_A i}})\circ \varphi^k_l - (-1)^s\varphi^k_{l+1}\circ\delta_M^l - (-1)^{l+s-k} b(\varphi^{k-1}_l).
\end{equation}
Equation~\eqref{d.phi} is obtained by translating to complexes of comodules the differential of $\ZZ$-connections through Proposition~\ref{prop.comod}. The novelty is that the definitions do not depend on the choice of an element of $\cC$.
An easy but rather lengthy calculation proves that the map $d$ is nilpotent and that the composition \eqref{comp} is a morphism of cochain complexes. 

Following \cite{BonKap:enh} or \cite{Blo:dua}, to any closed zero-degree morphism $\varphi: \cM\to \cN$ in \com \cC, one can associate its {\em cone}, $\mathrm{Cone}(\varphi) := (O^l,\varrho_l^O, \delta^l_O)$, where
$$
O^l = \begin{pmatrix} N^l \cr \oplus \cr M^{l+1} \end{pmatrix}, \qquad \varrho_l^O = \begin{pmatrix} \varrho^N_l  & \varphi^1_{l+1} \cr & & \cr 0 & \varrho^M_{l+1} \end{pmatrix}, \qquad \delta^l_O = \begin{pmatrix} -\delta_N^l  & \varphi^0_{l+1} \cr & & \cr 0 & \delta_M^{l+1} \end{pmatrix}.
$$
In this way the homotopy category of \com \cC~ becomes a triangulated category with distinguished triangles obtained from
$$
\xymatrix{ \cM \ar[r]^-\varphi & \cN \ar[r]^-{\iota} & \mathrm{Cone}(\varphi) \ar[r]^-\pi & \cM[1]},
$$
where $\cM[1]$ denotes the comodule complex obtained by shifting by one degrees in $\cM$ and the morphisms $\iota$ and $\pi$ have the only non-zero components
$$
\iota_l^0 (n) = (-1)^l \begin{pmatrix} n \cr 0 \end{pmatrix}, \qquad \pi_l^0  \begin{pmatrix} n \cr m \end{pmatrix} =m.
$$
Following \cite{Blo:dua} one can consider quasi-cohesive $\cC$-comodules $\cM:= (M^l, \varrho^M_l, \delta^l_M)$ that are bounded in both directions and such that each of the $M^l$ is  finitely generated and projective as an $A$-module, and term such objects {\em cohesive comodules}. The homotopy category of the category of cohesive comodules forms a full triangulated subcategory of the homotopy category of \com\cC.

\subsection{...and back}
Here we associate a based coring to  a semi-free curved differential graded algebra and thus establish an equivalence of categories of semi-free differential graded algebras and corings with surjective counits.
\begin{theorem}\label{thm.cdga-crg}
Let $A$ be an algebra, $V$ an $A$-bimodule and let $\cA = (T_A(V), d, \gamma)$ be a semi-free curved differential graded algebra. Consider the left $A$-module $C=A\oplus V$, and write $x$ for the generator of $C$ corresponding to $A$, so that a general element of $C$ can be written additively as $ax+v$, $a\in A$, $v\in V$.
\begin{zlist}
\item $\cC(A, V) = (C,\Delta,\eps)$ is an $A$-coring based at $x$ by the following operations:
\begin{blist}
\item The right $A$-multiplication of $a'\in A$ with $ax+v\in C$ is defined by:
$$
(ax +v) a' := aa'x +ada' + va';
$$
\item the comultiplication is defined by:
$$
\Delta(ax +v) = ax\ot x -a\gamma +x\ot v +v\ot x - dv;
$$
\item the counit is $\eps(ax+v) = a$.
\end{blist}
\item Any morphism $(f_0,f_1,\omega)$ of  curved differential graded algebras $(T_A(V), d_A, \gamma_A) \to (T_B(W), d_B, \gamma_B)$ yields a morphism $(\cC(f)_0,\cC(f)_1)$ of the corresponding (based) corings $\cC(A,V)\to  \cC(B,W)$, where
$$
\cC(f)_0 = f_0, \qquad \cC(f)_1: (a, v) \mapsto (f_0(a), f_0(a)\omega +f_1(v)).
$$
\item The assignments described in statements (1) and (2) define an equivalence of categories $\sfcdga \simeq \bcrg$.
\end{zlist}
\end{theorem}
\begin{proof}
(1) Since $d$ is a derivation, i.e.\ it satisfies the Leibniz rule, the right multiplication defined in (a) furnishes $C$ with the structure of an $A$-bimodule. Note that the right $A$-action is constructed in such a way that $da = [x,a]$, for all $a\in A$. Using this observation, take any $b\in A$ and compute
\begin{eqnarray*}
\Delta(b(ax +v)) &=& bax\ot x -ba\gamma +x\ot bv +bv\ot x - d(bv) \\
&=& b(ax\ot x -a\gamma +x\ot v +v\ot x - dv)
 + [x,b]\ot v - db\ot v \\
 &=& b \Delta(ax +v).
\end{eqnarray*}
Furthermore,
\begin{eqnarray*}
\Delta((ax +v)b) &=& \Delta(abx +adb + vb) \\
&= & abx\ot x -ab\gamma +x\ot vb  +vb\ot x - d(vb) + x\ot adb + adb\ot x - d(adb)\\
&=& abx\ot x -ab\gamma +x\ot vb  +v\ot bx - d(v)b +v\ot db\\
&& +  xa\ot db + adb\ot x - da\ot db - a[\gamma,b]\\
&=& ax\ot bx - a\gamma b + x\ot vb +v\ot xb - d(v)b + ax\ot db\\
&=& ax\ot xb - a\gamma b + x\ot vb +v\ot xb - d(v)b = \Delta(ax +v)b ,
\end{eqnarray*}
where the graded Leibniz rule and the curvature of $d$ were used in derivation of the third equality and we  applied  repeatedly that $d$ on $A$ is given by the commutator with $x$. This proves that $\Delta$ is an $A$-bimodule map.

Before we prove that $\Delta$ is coassociative, it is convenient to introduce a Sweedler-like notation for $\gamma$ and $dv$, i.e.\ to write $\gamma = \sum \gamma \su1 \ot \gamma\su 2$ and $dv = \sum dv\suc 1\ot dv\suc 2$. In this notation the Bianchi identity takes the form
$$
\sum d(\gamma\su 1) \ot \gamma\su 2 - \sum \gamma \su 1\ot d(\gamma\su 2) =0,
$$
while the curvature condition for $v$ reads
$$
\sum d(dv\suc 1)\ot dv\suc 2 - \sum dv\suc 1\ot d(dv\suc 2) + \sum v\ot \gamma\su 1\ot \gamma \su 2 - \sum  \gamma\su 1\ot \gamma \su 2 \ot v =0.
$$
Applying $(\id\ot \Delta - \Delta\ot \id)\circ \Delta$ to $ax +v$, and omitting obvious cancellations we obtain
\begin{eqnarray*}
&& -\sum ax\ot \gamma\su 1\ot \gamma \su 2 + \sum a \gamma\su 1\ot d(\gamma \su 2) - \sum v\ot \gamma\su 1\ot \gamma \su 2 + \sum dv\suc 1\ot d(dv\suc 2)  \\
&& + \sum xa\ot \gamma\su 1\ot \gamma \su 2 - \sum d(a\gamma\su 1)\ot \gamma \su 2 +\sum  \gamma\su 1\ot \gamma \su 2 \ot v - \sum d(dv\suc 1)\ot dv\suc 2,
\end{eqnarray*}
which vanishes by the curvature condition (terms 3, 4, 7 and 8) and the combination of the Leibniz rule (applied to the sixth term), the Bianchi identity and the fact that $da = [x,a]$. Therefore, $\Delta$ obeys the coassociative law. That $\eps$ is the counit for this comultiplication and that $x$ is a base point are immediate, since $\eps$ is designed so that it is identity on $A$ (i.e.\ $\eps(x)=1$) and zero on $V$.

(2) To facilitate the calculations we will write the bimodule parts of $\cC(A,V)$ and $\cC(B,W)$ additively as $Ax \oplus V$ and $By\oplus W$, respectively. We will also write $\Delta_C$ and $\eps_C$ for the comultiplication and counit of $\cC(A,V)$ and $\Delta_D$ and $\eps_D$ for the structure maps of $\cC(B,W)$.

First we need to check, whether $\cC(f)_1$ is an $A$-bimodule map. Since $f_0$ is an algebra map and $f_1$ is a left $A$-module map, also $\cC(f)_1$ is left $A$-linear. For the right $A$-linearity, take any $b\in A$, and compute
\begin{eqnarray*}
\cC(f)_1((ax +v)b) &=& f_0(ab)y +f_0(ab)\omega +f_1(ad_Ab + vb)\\
&=& f_0(a)f_0(b)y + f_0(a)f_0(b)\omega \\
&&+f_0(a)d_Bf_0(b) +f_0(a)[\omega, f_0(b)] + f_1(v)f_0(b)\\
&=& f_0(a)f_0(b)y + f_0(a)\omega f_0(b) +f_0(a)d_Bf_0(b) + f_1(v)f_0(b)\\
&=& \cC(f)_1(ax +v)f_0(b),
\end{eqnarray*}
where the condition (a) in Definition~\ref{def.cat} has been used. 

Clearly, $\eps_D\circ \cC(f)_1 = \cC(f)_0 \circ \eps_C$, hence $\cC(f)_1$ is compatible with counits, so only the preservation of comultiplications need be checked. Since the comultiplications and $\cC(f)_1$ are left $A$-module maps, suffices it to consider their actions on elements of the form $x + v$, $v\in V$. On one hand
\begin{eqnarray*}
\Delta_D (\cC(f)_1(x +v)) &=& y\ot y - \gamma_B + y\ot(\omega +f_1(v)) +(\omega +f_1(v))\ot y - d_B (\omega+f_1(v))\\
&=& (y+\omega)\ot (y +\omega) - (f_1\ot f_1)(\gamma_A)  \\
&&+ (y+\omega)\ot f_1(v) + f_1(v)\ot (y +\omega) - (f_1\ot f_1)(d_A v)\\
&=& \cC(f)_1(x)\ot \cC(f)_1(x) - (f_1\ot f_1)(\gamma_A)  \\
&&+ \cC(f)_1(x)\ot \cC(f)_1(v) + \cC(f)_1(v)\ot \cC(f)_1(x) - (f_1\ot f_1)(d_A v)\\
&=& (\cC(f)_1(v)\ot \cC(f)_1)(\Delta_C(x+v)),
\end{eqnarray*}
where the definition of a morphism of curved differential graded algebras was used to derive the second equality and the definition of $\cC(f)_1$ to derive the third and fourth ones. Therefore, $(\cC(f)_0,\cC(f)_1)$ is a morphism of corings as required.

(3) Consider the assignment:
$$
U : \sfcdga \to \bcrg , \qquad (T_A(V), d, \gamma)\mapsto \cC(A,V), \qquad f\mapsto (\cC(f)_0, \cC(f)_1).
$$
One easily checks that $U$ preserves the composition of morphisms, hence defines a functor. We will show that the functors $U$ and $T: (\cC,x)\mapsto (T_A(\cC^+), d_x,\gamma_x)$, constructed in Theorem~\ref{thm.crg-cdga}, are inverse equivalences. First  take a semi-free curved differential graded algebra $(T_A(V), d, \gamma)$. Since $(Ax\oplus V)^+ = V$, $T\circ U(T_A(V), d, \gamma) = (T_A(V), d_x, \gamma_x)$, where, for all $a\in A$ and $v\in V$,
$$
d_xa = xa - ax = da, \quad d_xv = x\ot v - \Delta(v) + v\ot x = dv,
$$
by the definitions of the right $A$-action and the comultiplication of $\cC(A,V)$. Thus $d_x = d$. Furthermore,
$$
\gamma_x = x\ot x - \Delta(x) = x\ot x -x\ot x +\gamma = \gamma, 
$$
by the definition of comultiplication of $\cC(A,V)$.
Therefore, $ T\circ U = \id$ on objects. Similarly, given a morphism $(f,\omega) : (T_A(V), d_A, \gamma_A) \to (T_B(W), d_B, \gamma_B) $, obviously $T\circ U(f)_0 =f_0$ and $T\circ U(f)_1 =f_1$. Furthermore, since $f_1 (1,0)  - (1,0) = (1,\omega) - (1,0) = (0,\omega)$, so that $T\circ U(f,\omega) = (f,\omega)$, hence $ T\circ U = \id$.

In the opposite direction, take an $A$-coring $\cC=(C,\Delta,\eps)$ based at $x\in C$, so that 
$$
U\circ T(\cC,x) = U(T_A(C^+), d_x, \gamma_x) = (Ay\oplus C^+, \Delta_V, \eps_V, y).
$$
Define 
$$\varphi: C\to Ay\oplus C^+, \qquad c\mapsto \eps(c)y + (c-\eps(c) x) = y\eps(c) + (c-x \eps(c)),
$$
where the equality follows by the definitions of the right $A$-multiplication on $Ay\oplus C^+$ and of $d_x$, i.e.\ $ya = ay +d_x a = ay + [x,a]$. We will prove that $(\id_A, \varphi)$ is an isomorphism of $A$-corings. In view of two equivalent descriptions of $\varphi$, $\varphi$ is an $A$-bimodule map. Clearly, $\eps_V\circ \varphi = \eps$. To prove that $\varphi$ preserves comultiplications, take any $c\in C$ and use the definitions of $\Delta_V$, $\gamma_x$ and $d_x$, and the fact that $[y,a] = [x,a]$, for all $a\in A$, to compute
\begin{eqnarray*}
\Delta_V(\varphi(c)) &=& \eps(c) y\ot y -\eps(c) x\ot x +\eps(c) \Delta(x) + y\ot (c - \eps(c)x) +(c - \eps(c)x)\ot y \\
&& -x\ot c + \Delta(c) - c\ot x +x\ot\eps(c)x - \eps(c) \Delta(x) +\eps(c)x\ot x\\
&=& y\eps(c)\ot  y + y\ot c -y\eps(c)\ot x + c\ot y\\
&& - x\ot c + \Delta(c) - c\ot x +x\eps(c)\ot x  -x\eps(c)\ot y \\
&=& \sum (y\eps(c\sw 1) + c\sw 1-x \eps(c\sw 1))\ot (\eps(c\sw 2)y + c\sw 2-\eps(c\sw 2) x)\\
& =&  (\varphi\ot\varphi)\circ \Delta(c).
\end{eqnarray*}
Hence $(\id_A, \varphi)$ is a morphism of based corings. Next define
$$
\psi: Ay\oplus C^+ \to C, \qquad ay + v\mapsto ax+v.
$$
Then, for all $c\in C$, 
$$
\psi\circ \varphi (c) = \eps(c)x + c - \eps(c)x  = c,
$$
and, for all $a\in A$, $v\in C^+$,
$$
\varphi\circ \psi(ay+v) = ay + ax +v  - \eps(ax +v) x = ay + ax +v -ax = ay +v.
$$
Therefore, $\varphi$ is an isomorphism with the inverse $\psi$. Finally, an equally straightforward calculation yields that $U\circ T$ is an identity on morphisms (and thus implies the naturality of the isomorphisms $\varphi$), and completes the proof of the assertion.
\end{proof}

\begin{corollary}\label{cor.surjective}
The category of semi-free curved differential graded algebras is equivalent to the full subcategory of the category of corings consisting of corings with surjective counit.
\end{corollary}
\begin{proof}
Combine the equivalence in Theorem~\ref{thm.cdga-crg}  with the equivalence in Lemma~\ref{base.sur}.
\end{proof}

\begin{remark}\label{rem.right}
The form of a correspondence between based corings and semi-free curved differential graded algebras relates the kernel of a counit with the bimodule on which the differential algebra is built, so it should not depend on the direct sum decomposition of a coring into the kernel of the counit and the base ring. Yet Theorem~\ref{thm.cdga-crg}  describes a specific left $A$-module decomposition. In fact, given a semi-free curved differential graded algebra $\cA = (T_A(V), d, \gamma)$  one can construct an $A$-coring $\cD(A, V) = (D,\Delta,\eps)$ based at, say, $y$ as follows. $D = A\oplus V$ as a {\em right} $A$-module, so we will write $ya +v$ for an element of $D$. The left $A$-multiplication is defined by 
$$
a(yb +v) = yab -d(a)b +av,
$$
while the coproduct and counit are:
$$
\Delta(ya +v) = y\ot ya - \gamma a +y\ot v + v\ot x -dv, \qquad \eps(ya+v) =a.
$$
One can easily check, however, that  $\cD(A, V)$ is isomorphic to $\cC(A, V)$ by the map of corings given as the identity on $A$ and on $D$ as
$$
D\to C, \qquad ya +v \mapsto ax +da +v = xa +v,
$$
(since $da = [x,a]$ in $C$). Thus the choice made in establishing an equivalence in Theorem~\ref{thm.cdga-crg}  is a matter of convention not essence.
\end{remark}

One can complement  Proposition~\ref{prop.comod} by assigning comodule complexes to integrable $\ZZ$-connections. Recall from Example~\ref{ex.connection} that the covariant derivative of a $\ZZ$-connection $(M^\bullet, \cov)$ is fully determined by its value on $M^\bullet = M^\bullet \ot_A A^0$ (recall that $A=A^0$). As before, we write $\cov^{n,l}$ for the component of $\cov$ given as $\cov^{n,l}: M^l \mapsto M^{l -n+1} \ot_A A^n$.

 \begin{proposition}\label{prop.con}
 Let $(M^\bullet, \cov)$ be an integrable $\ZZ$-connection over a semi-free curved differential algebra $(T_A(V), d,\gamma)$, such that $\cov^{n,l} =0$, for all $n\geq 2$. Write $C= Ax\oplus V$ and define, for all $l\in \ZZ$,
 $$
 \varrho_l: M^l \to M^l\ot_A (A\oplus V), \qquad m\mapsto (-1)^l\cov^1(m) + m\ot x.
 $$
 Then $(M^l, \varrho_l, \cov^{0,l})$ is a complex of right comodules over the based coring $U(T_A(V), d,\gamma)$.
 \end{proposition}
 \begin{proof}
 By the definition of the right $A$-multiplication on $Ax\oplus V$, Theorem~\ref{thm.cdga-crg}(1)(a), $da = [x,a]$, hence, since $(-1)^l\cov^{1,l}$ satisfies the Leibniz rule,
$$
\cov^{1,l}(ma) = \cov^1(m)a + (-1)^lm\ot xa - (-1)^lma \ot x,
$$
which in turn implies that the $\varrho_l$ are right $A$-linear maps. In view of the definition of the counit of $U(T_A(V), d,\gamma)$, the $\varrho_l$ are evidently counital. Finally, the integrability of $\cov$ yields three conditions:
$$
\cov^{0,l+1}\circ \cov^{0,l} =0, \qquad (\cov^{0,l}\ot \id)\circ \cov^{1,l} + \cov^{1,l+1}\circ \cov^{0,l}=0,
$$
and, for all $m\in M^l$, 
$$
(\cov^{1,l} \ot_A\id  )\circ \cov^1(m) + (-1)^l (\id \ot_A d) \circ \cov^{1,l}(m) + m\ot \gamma =0,
$$
Combining the third identity with the definition of the comultiplication of $U(T_A(V), d,\gamma)$ in Theorem~\ref{thm.cdga-crg}(1)(b), one easily verifies that the $\varrho_l$ satisfy the coassociative law. 
The first two conditions then mean that $(M^l, \varrho_l, \cov^{0,l})$ is a complex of comodules over $U(T_A(V), d,\gamma)$. 
 \end{proof}

\section{Integrable divergences}
\setcounter{equation}{0}
The aim of this section is to introduce the notion of an {\em integrable $\ZZ$-divergence} as a special example of a curved module that supplements the notion of a quasi-cohesive module. 

Consider a curved differential graded algebra $\cA = (A^\bullet, d,\gamma)$ and a $\ZZ$-graded right $A$-module $M^\bullet$. Construct the  $\ZZ$-graded right $A^\bullet$-module $\Xi(A,M)^\bullet$ by first setting
$$
\Xi(A,M)^n = \prod_{i\in \NN}\rhom A {A^{i}}{M^{n+i}},
$$
and then defining the right $A^\bullet$-multiplication as follows. Write  $\xi\in \Xi(A,M)^n$  as  $\xi = (\xi_i)_{i\in \NN} $, where $\xi_i\in \rhom A {A^i} {M^{n+i}}$. Then, for all $a\in A^k$, define $\xi_i a\in \rhom A{A^{i-k}} {M^{n+i}}$ by 
$$
\xi_i a = \begin{cases} b\mapsto \xi_i(ab), & \mbox{if $i\geq k$,}\\
0, & \mbox{otherwise,}
\end{cases}
$$
and set $\xi a := (\xi_ia)_{i\in \NN}$.

\begin{definition}\label{def.hom-conn}
Let $\cA = (A^\bullet, d,\gamma)$ be a curved differential graded algebra, and write $A :=A^0$. A {\em $\ZZ$-divergence} over $\cA$ is a pair $(M^\bullet, \voc)$ consisting of a $\ZZ$-graded right $A$-module $M^\bullet$ and the degree-one map $\voc: \Xi(A,M)^\bullet\to \Xi(A,M)^\bullet$ such that, for all homogeneous $\xi\in \Xi(A,M)^\bullet$ and $a\in A^\bullet$,
$$
\voc (\xi a) = \voc(\xi)a + (-1)^{|\xi |} \xi da.
$$
A $\ZZ$-divergence $(M^\bullet, \voc)$ is said to be {\em integrable} provided, for all $\xi\in \Xi(A,M)^\bullet$,
$$
\voc\circ\voc (\xi) = - \xi\gamma.
$$
\end{definition}
Clearly, $(M^\bullet, \voc)$ is an integrable $\ZZ$-divergence over $\cA$ if and only if $(\Xi(A,M)^\bullet, \voc)$ is a curved right $\cA$-module.

\begin{remark}\label{rem.div}
The map $\voc$ is fully determined by its components
$$
\xymatrix{\voc^{n,m}:
 \Xi(A,M)^{n} \ar[r]^-{\voc} &  \Xi(A,M)^{n+1} \ar@{->>}[r]^-{\pi} & \rhom A {A^m} {M^{m+n+1}},}
$$
where $\pi$ is the canonical epimorphism in the direct product decomposition of $\Xi(A,M)^\bullet$. In terms of these components, the Leibniz rule in Definition~\ref{def.hom-conn} comes out as follows. For all $a\in A^s$ and  all $\xi = (\xi_k)_{k\in \NN} \in  \prod_{k\in \NN}\rhom A {A^{k}}{M^{n+k}}$,
\begin{equation}\label{nablanm}
\voc^{n+s,m} (\xi a) = \voc^{n, m+s}(\xi)a + (-1)^n \xi_{m+s+1} da.
\end{equation}
Next, define $\voc_k^n: \rhom A {A^k}{M^{n +k }}\to  M^{n+1}$ as  the composite of $\voc^{n,0}$  with the inclusion $\rhom A {A^k}{M^{n+k}} \hookrightarrow \Xi(A,M)^{n}$. Then Equation~\eqref{nablanm} implies, for all $a\in A$ and $\xi \in \rhom A {A^k}{M^{n +k }}$,
\begin{equation}\label{nablan0}
\voc^n_k(\xi a) = \voc^n_k(\xi)a +(-1)^n\delta_{k1}\xi da .
\end{equation}
Therefore, for all  $k\neq 1$, the $\voc_k^n$ are $A$-linear, while  $(M^n, (-1)^{n}\voc_1^n)$ are hom-connections in the sense of \cite{Brz:con}.

Conversely, given a family of maps $\voc^n_k$, finite for each $n$ and  satisfying conditions in Equation~\eqref{nablan0}, one can define $\voc^{n,m}: \Xi(A,M)^{n} \to  \rhom A {A^m} {M^{m+n+1}}$, by setting, for all $\xi = (\xi_k)_{k\in \NN} \in  \prod_{k\in \NN}\rhom A {A^{k}}{M^{n+k}}$ and $a\in A^m$,
\begin{equation}\label{0tom}
\voc^{n,m}(\xi)(a) = \sum_k\voc^{n+m}_{k}(\xi_{m+k} a) + (-1)^{n+1}\xi_{m+1}(da).
\end{equation}
These will satisfy Equations~\ref{nablanm} and thus can be used to construct a $\ZZ$-divergence.
\end{remark}

\begin{example}\label{ex.contra}
Let $\cA = (A^\bullet, d,\gamma)$ be a curved differential graded algebra. For any curved $\cA$-module $(M^\bullet ,d_M)$ define
$$
\voc :  \Xi(A,M)^{n} \to  \Xi(A,M)^{n+1}, \qquad \xi\mapsto d_M\circ \xi - (-1)^n \xi\circ d.
$$
Then  $(M^\bullet ,\voc)$ is an integrable $\ZZ$-divergence.
\end{example}
\begin{proof}
Set $(E^\bullet, d_E)$ to be $(A^\bullet, d)$ in Lemma~\ref{lem.ind}.
\end{proof}

\begin{lemma}\label{lem.hom}
Let $\cA = (A^\bullet, d,\gamma)$ be a curved differential graded algebra, $B$ an algebra and let $M^\bullet$ be a $\ZZ$-graded $(A,B)$-bimodule, bounded from below.  Let $(M^\bullet, \cov)$ be a {\em left} $\ZZ$-connection over $\cA$ such that $\cov$ is right $B$-linear. Write $\cov^k$ for the component $M^\bullet \to  A^k\ot_A M^{\bullet - k+1}$ of $\cov$, and $\cov^{k,n}$ for the restriction of $\cov^k$ to $M^n$. Denote by $N^\bullet$ the graded $B$-dual of $M^\bullet$, i.e. $N^{n} = \rhom B {M^{-n}} B$. Then the family
$$
\voc_k^{-n+1} :=  \rhom B {\cov^{k,n}} B, \qquad k\in \ZZ
$$ 
determines a $\ZZ$-divergence on $N^\bullet$.
\end{lemma}
\begin{proof}
By the hom-tensor relations, the domain of $ \rhom B {\cov^{k,n}} B$ is
$$
\rhom B {A^k\ot_A M^{n - k+1}} B   \cong \rhom A {A^k}{\rhom B {M^{n - k+1}} B }  = \rhom A {A^k}{N^{-n + k-1}}, 
$$
while the codomain is $\rhom B {M^{n}} B = N^{-n}$, as required.

Since $\cov$ is a left covariant derivative, the maps $\cov^{k\neq 1}$ are left $A$-linear, while for all $m\in M^n$, $a\in A$,
$$
\cov^{1,n} (am) = a\cov^{1,n}(m) + (-1)^nda\ot m.
$$
Thus $\voc_{k\neq 1}^n$ are right $A$-linear and, for all $\xi \in \rhom A {A^1} {N^{n+1}} \cong \rhom B {A^1 \ot M^{-n-1}} B$, $a\in A$ and $m\in M^{-n-1}$,
\begin{eqnarray*}
\voc_1^n(\xi a)(m) &=& \xi a (\cov^{1,-n-1}(m)) = \xi(a \cov^{1, -n-1}(m)) = \xi (\cov^1 (am)) + (-1)^n\xi (da\ot m)\\
& = &\voc^n_1(\xi)(am) + (-1)^n\xi (da)(m) =
(\voc_1^n(\xi)a+ (-1)^n\xi (da))(m).
\end{eqnarray*}
Therefore, the $\voc_k^n$ define a $\ZZ$-divergence over $\cA$. 
\end{proof}

\begin{example}\label{ex.div.comatrix}
Let $P$ be a progenerator right $A$-module, take $B\subset S = \rend A P$ and consider the curved differential graded algebra $\cA({}_BP_A)$ described in Example~\ref{ex.comatrix}. As in Example~\ref{ex.comatrix}, denote by $e\in P \ot_A P^*$ the dual basis element and take $x= \sum_i \xi_i \ot x_i \in P^*\ot _B P$, an element in the inverse image of $1_A$ under the evaluation map. $P^*$ is a left comodule over the comatrix coring $P^*\ot _B P$, and thus $(P^*, \cov)$ is a left cohesive module over $\cA({}_BP_A)$, with
$$
\cov^1 : P^* \to P^*\ot_B A^1, \qquad \chi\mapsto  x\ot \chi -\chi \ot e .
$$
Since, for all $b\in B$, $be=eb$, the map $\cov^1$ is right $B$-linear, hence by Lemma~\ref{lem.hom} it induces a $\ZZ$-divergence over $\cA({}_BP_A)$ with the module part $N = \rhom B {P^*} B$ and with
$
\voc_1 : \rhom A {A^1 \ot_A P^*} B \cong \rhom A {A^1} N \to N$ given by $\voc_1(f)(\chi) = f(x\ot \chi - \chi\ot e)$.
\end{example}

\begin{proposition}\label{prop.contra-hom}
Let $\cC = (C,\Delta, \eps)$ be an $A$-coring and let $x\in C$. Given a complex of $C$-contramodules $\delta^n: (M^n,\alpha_n) \to (M^{n+1},\alpha_{n+1})$, define
$$
\voc_0^n = \delta^n, \qquad \voc_1^n: \rhom A C {M^{n+1}} \to M^{n+1}, \qquad \xi \mapsto (-1)^n\left( \xi(x) - \alpha_{n+1}(\xi)\right),
$$ 
and $\voc^n_k =0$, for $k\geq 2$. Then the $\voc_k^n$ define an integrable $\ZZ$-divergence $(M^\bullet,\voc)$ with respect to curved differential graded algebra $T^\flat(\cC, x)$, as described in Remark~\ref{rem.div}. Furthermore, if $\eps(x) =1$, then 
$$
\xymatrix{
\voc^{+n}_1 : \rhom A {C^+} {M^{n+1}} \ar[rrr]^-{\rhom A {\pi_x^R} {M^{n+1}}} &&& \rhom A C {M^{n+1}}\ar[rr]^-{\voc_1^n} &&  M^{n+1},}
$$
where $\pi_x^R: C\to C^+$, $c\mapsto c- x\eps(c)$, define an integrable $\ZZ$-divergence with respect to $T(\cC, x)$.
\end{proposition}
\begin{proof}
Using the definition of $d_x$ and the right $A$-linearity of $\alpha_n$ one easily checks that the $\voc_1^n$ satisfy Equation~\eqref{nablan0}. Furthermore $\nabla_{k\neq 1}^n$ are all $A$-linear, hence they all induce a $\ZZ$-divergence, with components given by Equation~\eqref{0tom}, 
\begin{eqnarray*}
\voc^{n,m}(\xi)(c) &= &\delta^{m+n}(\xi_m(c)) 
- (-1)^{n}\xi_{m+1}(x\ot c) \\
&&+ \sum_{k=1}^{m}(-1)^{k+n+1}\xi_{m+1}\left(\Delta^{(m)}_k\left(c\right)\right) + (-1)^{m+n}\alpha_{m+n+1}(\xi_{m+1} c),
\end{eqnarray*}
for all $\xi  = (\xi_k) \in \prod_{k} \rhom A {C^{\otimes _Ak}} {M^{n+k}}$ and $c\in C^{\otimes _Am}$. A rather lengthy but straightforward calculation, not dissimilar to that in the proof of Proposition~\ref{prop.comod}, which uses the associativity of contra-actions $\alpha_n$, that $\delta^{n+1}\circ \delta^n =0$ and that each of the $\delta^n$ is a contramodule map, as well as the fact that $d_x\circ d_x(c) = [\gamma_x,c]$ and the Leibniz rule allows one to conclude that
$$
\voc^{n+1,m} (\voc^n(\xi))(c) = -(\xi_{m+2}\gamma_x)(c),
$$
for all $\xi  = (\xi_k) \in \prod_{k} \rhom A {C^{\otimes _Ak}} {M^{n+k}}$ and $c\in C^{\otimes _Am}$. Therefore, $\voc \circ \voc (\xi) = -\xi \gamma_x$, as required for an integrable $\ZZ$-divergence.

The second assertion is a consequence of the first which can be seen as follows. Take $\xi \in \rhom A {C^+} {M^{n+1}}$, and let $\overline{\xi} := \xi \circ \pi_x^R$, so that $\voc^{+n}_1(\xi) = \voc_1^n(\overline{\xi})$. Note that $\overline{\xi} \mid_{C^+} = \xi$, so, in particular, $\overline{\xi}(d_xa) = \xi(d_xa)$. Note further, that
$
\overline{\xi a} = \overline{\xi} a +\xi(d_xa)\eps.
$
Hence,
\begin{eqnarray*}
\voc^{+n}_1(\xi a) &=& \voc_1^n(\overline{\xi a}) = \voc_1^n(\overline{\xi} a) +\voc_1^n(\xi(d_xa)\eps)\\
&=& \voc_1^n(\overline{\xi}) a + (-1)^n\xi(d_xa) +(-1)^{n+1}\xi(d_xa)\eps(x) + (-1)^n\alpha_{n+1}(\xi(d_xa)\eps) \\
&=& \voc^{+n}_1(\xi) a + \xi(d_xa),
\end{eqnarray*}
by the counitality of the $\alpha_n$ and the fact that $\eps(x)=1$. Thus the $\voc_1^{+n}$ induce a $\ZZ$-divergence as claimed. The integrability follows by the same arguments as for $\voc$.
\end{proof}

\end{document}